\documentclass[11pt]{amsart}
\usepackage[colorlinks,citecolor=red, dvipdfm]{hyperref}

\newtheorem{theorem}{Theorem}[section]
\newtheorem{lemma}[theorem]{Lemma}

\theoremstyle{definition}

\newtheorem{remark}[theorem]{Remark}

\theoremstyle{remark}

\newcommand{\bs}{\begin{split}}
\newcommand{\es}{\begin{split}}

\newcommand{\be}{\begin{equation}}
\newcommand{\ee}{\end{equation}}

\numberwithin{equation}{section}

\begin{document}

\title[Uniqueness of the complex projective spaces]
{Some remarks on the uniqueness of the complex projective spaces}

\author{Ping Li}
\address{Department of Mathematics, Tongji University, Shanghai 200092, China}
\email{pingli@tongji.edu.cn\\
pinglimath@gmail.com}
\thanks{The author was partially supported by the National
Natural Science Foundation of China (Grant No. 11471247) and the
Fundamental Research Funds for the Central Universities.}




\begin{abstract}
We first notice in this article that if a compact K\"{a}hler
manifold has the same integral cohomology ring and Pontrjagin
classes as the complex projective space $\mathbb{C}P^n$, then it is
biholomorphic to $\mathbb{C}P^n$ provided $n$ is odd. The same holds
for even $n$ if we further assume that $M$ is simply-connected. This
technically refines a classical result of Hirzebruch-Kodaira and
Yau. This observation, together with a result of Dessai and Wilking,
enables us to characterize all $\mathbb{C}P^n$ in terms of homotopy
type under mild symmetry. When $n=4$, we can drop the requirement on
Pontrjagin classes by showing that a simply-connected compact
K\"{a}hler manifold having the same integral cohomology ring as
$\mathbb{C}P^4$ is biholomorphic to $\mathbb{C}P^4$, which improves
on results of Fujita and Libgober-Wood.
\end{abstract}

\maketitle
\section{Introduction}
It is an important problem to characterize the standard complex
projective spaces $\mathbb{C}P^n$ as compact complex manifolds via
geometrical and/or topological information as little as possible.
Hirzebruch and Kodaira showed in \cite{HK} that if a K\"{a}her
manifold $M$ is diffeomorphic to $\mathbb{C}P^n$, then
\begin{enumerate}
\item
$M$ is biholomorphic to $\mathbb{C}P^n$ provided that $n$ is odd;

\item
$M$ is biholomorphic to $\mathbb{C}P^n$ provided that $n$ is even
and the first class $c_1(M)\neq-(n+1)g$, where $g$ is the positive
generator of $H^2(M;\mathbb{Z})$.
\end{enumerate}

The fact that the total Pontrjagin class of $M$ has the standard
form $(1+g^2)^{n+1}$ as that of $\mathbb{C}P^n$ plays a key role in
their proof. Later Yau noticed that (\cite{yau}) the extra
assumption $c_1(M)\neq-(n+1)g$ in the case of $n$ being even can be
removed by Yau's Chern number inequality and the hypothesis
``diffeomorphic" can be relaxed to ``homeomorphic" due to Novikov's
result that the rational Pontrjagin classes are indeed homeomorphism
invariants (\cite{No}). In summary, we have
\begin{theorem}[Hizebruch-Kodaira \cite{HK}, Yau \cite{yau}]\label{HKY}
If a K\"{a}hler manifold is homeomorphic to $\mathbb{C}P^n$, it must
be biholomorphic to $\mathbb{C}P^n$.
\end{theorem}
 In order to deduce their main result in \cite{HK},
Hirzebruch and Kodaira showed a related result, \cite[p. 210,
Theorem 6]{HK}. Inspired by the idea of the arguments of
\cite[Theorem 6]{HK}, Kobayashi and Ochiai gave in \cite{KO} another
characterization of $\mathbb{C}P^n$ as well as a characterization of
hyperquadrics in terms of the Fano index of a Fano manifold. Recall
that a compact complex manifold $M$ is called \emph{Fano} if its
first Chern class $c_1(M)$ is positive. By Kodaira's embedding
theorem a Fano manifold is projective and thus automatically
K\"{a}hler. The \emph{Fano index} of a Fano manifold $M$ is defined
to be the largest positive integer $I$ such that $c_1(M)/I\in
H^2(M;\mathbb{Z})$. We denote by $I(M)$ the Fano index of $M$.
Kobayashi and Ochiai showed in \cite{KO} that if the Fano index of
an $n$-dimensional Fano manifold is no less than $n+1$, it must be
biholomorphic to $\mathbb{C}P^n$. A later result of Michelsohn (cf.
\cite[p. 366]{LM} or \cite[p. 1143]{Mi}) indeed showed that the Fano
index of a Fano manifold can not be larger than $n+1$. We summarize
them into the following
\begin{theorem}[Kobayashi-Ochiai \cite{KO}, Michelsohn \cite{Mi}]\label{KOM}
 Suppose $M$ is a Fano manifold. Then $I(M)\leq n+1$, with equality if and only if
$M$ is biholomorphic to $\mathbb{C}P^n$.
\end{theorem}

\begin{remark}A recent exposition paper \cite{To} by Tosatti presents a
detailed proof of Theorem \ref{HKY} and some results in \cite{KO} as
well as some necessary background knowledge. Moreover, he gave a
detailed proof of the fact that the nonexistence of exotic complex
structures on $\mathbb{C}P^3$ implies the nonexistence of complex
structures on $S^6$ (\cite[Prop. 3.1]{To}), which was originally
observed by Hirzebruch (\cite[p. 223]{Hi1}).
\end{remark}

The next natural question is whether we are able to relax the
hypotheses ``K\"{a}hlerness" and ``homeomorphism" in Theorem
\ref{HKY} to guarantee that its conclusion remains true. For general
$n$, we have no essentially stronger results up to now, at least to
the author's best knowledge. But when $n$ are small enough, we
indeed have some stronger results. For $n=2$, still applying his
Chern number inequality, together with some well-known facts on
compact complex surfaces, Yau showed that (\cite{yau}) a compact
complex surface homotopy equivalent to $\mathbb{C}P^2$ is
biholomorphic to $\mathbb{C}P^2$, which also solved an old
conjecture in algebraic geometry posed by Severi. For $n=3$, Lanteri
and Struppa showed that (\cite{LS}) a compact K\"{a}hler threefold
having the same integral cohomology ring as $\mathbb{C}P^3$ is
biholomorphic to $\mathbb{C}P^3$, in whose proof Yau's Chern number
inequality is still a major ingredient. For $n=4$ or $5$, by
applying Theorem \ref{KOM}, Fujita showed that a Fano manifold
having the same integral cohmology ring as $\mathbb{C}P^4$ or
$\mathbb{C}P^5$ is biholomorphic to $\mathbb{C}P^4$ or
$\mathbb{C}P^5$. By applying Theorem \ref{KOM} and a formula
relating the Chern number $c_1c_{n-1}$ to Hodge numbers discovered
by themselves in \cite{LW}, Libgober and Wood showed that if a
compact K\"{a}hler manifold is homotopically equivalent to
$\mathbb{C}P^n$ for $n=4$, $5$ or $6$, then it is biholomorphic to
$\mathbb{C}P^n$. We collect the above-mentioned results into the
following

\begin{theorem}\label{smallvalue}
We denote by $(S)$ the following statement:
$$(S):=\text{A compact complex manifold $M$ is biholomorphic to $\mathbb{C}P^n$}.$$
Then
\begin{enumerate}
\item
\text{(Yau,\cite{yau})} When $n=2$, $(S)$ holds if we assume that
$M$ is homotopically equivalent to $\mathbb{C}P^n$;

\item
\text{(Lanteri-Struppa,\cite{LS})} When $n=3$, $(S)$ holds if we
assume that $M$ is K\"{a}hler and has the same integral cohomology
ring as $\mathbb{C}P^n$;

\item
\text{(Fujita,\cite{Fu})} When $n=4$ or $5$, $(S)$ holds if we
assume that $M$ is Fano and has the same integral cohomology ring as
$\mathbb{C}P^n$;

\item
\text{(Libgober-Wood,\cite{LW})} When $n=4$, $5$ or $6$, $(S)$ holds
if we assume that $M$ is K\"{a}hler and homotopically equivalent to
$\mathbb{C}P^n$.
\end{enumerate}
\end{theorem}

\begin{remark}\label{remark}
Note that, when $n=4$ or $5$, the assumptions in $(3)$ and $(4)$ of
Theorem \ref{smallvalue} can not imply each other and thus their
results are independent. Also note that the proof in \cite{Fu} is
sketchy and many details were omitted.
\end{remark}

\section{Main observations}
We shall present in this section our main observations of this
article, Theorems \ref{refined}, \ref{homotopytorus} and
\ref{value4} and postpone their proofs to the next section.

Our first observation is that, if we combine some of Hirzebruch and
Kodaira's original arguments in \cite{HK} and Kobayashi-Ochiai's
criterion in Theorem \ref{KOM}, the original hypothesis
``homeomorphism" in Theorem \ref{HKY} can be relaxed to yield the
following
\begin{theorem}\label{refined}
Suppose $M$ is a compact K\"{a}hler manifold having the same
integral cohomology ring and Pontrjagin classes as $\mathbb{C}P^n$.
Then
\begin{enumerate}
\item
$M$ is biholomorphic to $\mathbb{C}P^n$ provided that $n$ is
odd;

\item
$M$ is biholomorphic to $\mathbb{C}P^n$ provided that $n$ is even
and $M$ is simply-connected.
\end{enumerate}
\end{theorem}

\begin{remark}~
\begin{enumerate}
\item
Since $H^{\ast}(\mathbb{C}P^n;\mathbb{Z})$ has no torsion, rational
Pontrjagin classes coincide with integral Pontrjagin classes and
thus our hypotheses in Theorem \ref{refined} is strictly weaker than
those in Theorem \ref{HKY}.

\item
It must be known to some experts that the original hypothesis
``homeomorphism" in Theorem \ref{HKY} can be relaxed to some extent.
However, to the author's best knowledge, there is no literature
where this assumption was explicitly refined in the form as in our
Theorem \ref{refined}.
\end{enumerate}
\end{remark}

In view of Theorem \ref{KOM}, in order to complete the proof of
Theorem \ref{refined}, it suffices to show that $c_1(M)=(n+1)g$ with
$g$ being a positive generator of $H^2(M;\mathbb{Z})$. We shall see
in the next section in this process the invariance of Pontrjagin
classes play a key role. However, as we have mentioned in Theorem
\ref{smallvalue}, when $n\leq6$, only assuming homotopy equivalence
and without requirement on Pontrjagin classes, Libgober and Wood can
still be able to show that $c_1(M)=(n+1)g$ by applying some subtle
invariants of homotopy equivalence in algebraic topology. But their
methods are ad hoc and fail to treat the general $n$. Our second
observation is that, if we allow the manifold $M$ to have mild
symmetry, the same result still holds for general $n$.

A smooth closed $2n$-dimensional manifold is called an
\emph{$n$-dimensional homotopy complex projective space} if it is
homotopically equivalent to $\mathbb{C}P^n$. Recall that a classical
conjecture in transformation group theory, which was posed by Petrie
in \cite{Pe1} and is still open in its full generality, asserts that
if an $n$-dimensional homotopy complex projective space $M$ admits
an (effective and smooth) circle action, then its total Pontrjagin
class agrees with that of $\mathbb{C}P^n$, i.e.,
$p(M)=(1+g^2)^{n+1}$ for a generator $g$ of $H^2(M;\mathbb{Z})$.
Petrie himself verified this conjecture (\cite{Pe2}) under the
stronger hypothesis that if an $n$-dimensional torus acts
(effectively and smoothly) on $M$. Dessai and Wilking improved on
Petrie's result by showing that the conjecture holds if  a torus
whose dimension is larger than $\frac{n+1}{4}$ acts on $M$ (\cite[p.
506]{DW}). Now combining Theorem \ref{refined} with Dessai-Wilking's
this result, our second observations reads
\begin{theorem}\label{homotopytorus}
If a compact K\"{a}hler manifold is homotopically equivalent to
$\mathbb{C}P^n$ and acted on effectively and smoothly by a torus
whose dimension is larger than $\frac{n+1}{4}$, then it must be
biholomorphic to $\mathbb{C}P^n$. When $n\leq 6$, the latter
hypothesis can be dropped by various results in Theorem
\ref{smallvalue}.
\end{theorem}

We now turn to our third observation in this article. We have
mentioned in Remark \ref{remark} that, when $n=4$, the hypotheses of
Fujita and Libgober-Wood can not imply each other and thus are
independent. Our third observation is to present a weaker hypothesis
than both of them. As is now well-known that a Fano manifold is
simply-connected, which is a corollary of the celebrated Calabi-Yau
theorem (cf. \cite[p. 225]{Zh}), the conditions of
simply-connectedness and having the same integral cohomology ring
are strictly weaker than the assumptions in $(3)$ and $(4)$ of
Theorem \ref{smallvalue}. Therefore our third observation, which
improves on the results of Fujita and Libgober-Wood when $n=4$,
asserts that the assumption on the invariance of Pontrjagin classes
in Theorem \ref{refined} can be dropped if $n=4$:
\begin{theorem}\label{value4}
A simply-connected compact K\"{a}hler manifold having the same
integral cohomology ring as $\mathbb{C}P^4$ is biholomorphic to
$\mathbb{C}P^4$.
\end{theorem}

\section{Proofs of Theorems \ref{refined} and \ref{value4}}
\subsection{Proof of Theorem \ref{refined}}
We first show the following key lemma under the hypotheses in
Theorem \ref{refined} , which is \cite[p. 208, Lemma2]{HK}. Here the
idea of our proof was still adopted from \cite{HK} but is more
direct and compact. We shall also see from this process that the
technical assumption we need is only the invariance of the integral
cohomology ring and Pontrjagin classes.
\begin{lemma}
Suppose $M$ is a compact K\"{a}hler manifold and its integral
cohomology ring and Pontrjagin classes are the same as those of
$\mathbb{C}P^n$. Then $c_1(M)=(n+1)g$ (resp. $c_1(M)=\pm(n+1)g$)
provided $n$ is odd (resp. even). Here $g$ is the positive generator
of $H^2(M;\mathbb{Z})$.
\end{lemma}

\begin{proof}
By assumptions we have
$$H^{\ast}(M;\mathbb{Z})=\mathbb{Z}[g]/(g^{n+1}), \qquad\int_Mg^n=1,$$
 and
the total Pontrjagin class of $M$ is given by $p(M)=(1+g^2)^{n+1}$.
We first note that the Hodge numbers of $M$ are the same as those of
$\mathbb{C}P^n$. Indeed, the famous relations of Hodge numbers for
compact K\"{a}hler manifolds tell us that
$$h^{p,p}(M)\geq 1,\qquad h^{p,q}(M)\geq0\qquad \text{for $0\leq p,q\leq n,$}$$
and
$$\sum_{p+q=i}h^{p,q}(M)=\text{the $i$-th Betti number of $M$}=
\text{the $i$-th Betti number of
$\mathbb{C}P^n$}=\frac{1+(-1)^i}{2},$$ which imply that
$$h^{p,p}(M)=1~(0\leq p\leq n),\qquad\text{and}\qquad h^{p,q}(M)=0~(p\neq q).$$
These lead to the value of the Todd genus of $M$, $\text{td}(M)$,
via
$$\text{td}(M)=\sum_{q=0}^n(-1)^qh^{0,q}(M)=1.$$
The Todd genus is a complex genus in the sense of Hirzebruch
(\cite[Ch.1,3]{Hi} or \cite[\S 1.8]{HBJ}) whose associated power
series is
$$\frac{x}{1-e^{-x}}=e^{\frac{x}{2}}\cdot\frac{x}
{e^{\frac{x}{2}}-e^{-\frac{x}{2}}}.$$ Note that $\frac{x}
{e^{\frac{x}{2}}-e^{-\frac{x}{2}}}$ is nothing but the even power
series whose associated genus is the $\hat{A}$-genus and can be
defined for oriented closed smooth manifolds in terms of Pontrjagin
classes (\cite[Ch.1]{Hi},\cite[\S 1.6]{HBJ}). We now suppose
$c_1(M)=kg$ for $k\in\mathbb{Z}$. In view of the fact that
$p(M)=(1+g^2)^{n+1}$, we have \be\begin{split}
 1=\text{td}(M)&=\int_Me^{\frac{kg}{2}}\cdot
 (\frac{g}{e^{\frac{g}{2}}-e^{-\frac{g}{2}}})^{n+1}\\
&=\int_Me^{\frac{(k+n+1)g}{2}}\cdot
 (\frac{g}{e^{g}-1})^{n+1}\\
 &=\text{the coefficient of $g^n$ in
 $e^{\frac{(k+n+1)g}{2}}\cdot
 \frac{g^{n+1}}{(e^{g}-1)^{n+1}}$}\\
 &=\text{the residue of $e^{\frac{(k+n+1)g}{2}}\cdot
\frac{1}{(e^{g}-1)^{n+1}}$ at $g=0$}\\
&=\frac{1}{2\pi\sqrt{-1}}\oint e^{\frac{(k+n+1)g}{2}}\cdot
\frac{1}{(e^{g}-1)^{n+1}}\text{d}g\\
&=\frac{1}{2\pi\sqrt{-1}}\oint
\frac{(y+1)^{\frac{k+n-1}{2}}}{y^{n+1}}\text{d}y\qquad
(e^g-1=:y)\\
&=\text{the coefficient of $y^n$ in $(y+1)^{\frac{k+n-1}{2}}$}\\
&=\frac{j(j-1)\cdots(j-n+1)}{n!}.\qquad (\frac{k+n-1}{2}=:j)
\end{split}\nonumber\ee
This yields $$n!=j(j-1)\cdots(j-n+1)$$ and thus $j=n$ (resp. $j=n$
or $-1$) provided $n$ is odd (resp. even). This implies $k=n+1$
(resp. $k=\pm(n+1)$) provided $n$ is odd (resp. even) and thus
completes the proof.
\end{proof}

In view of Theorem \ref{KOM}, in order to complete the proof of
Theorem \ref{refined}, it suffices to show that, when $n$ is even,
the additional hypothesis of simply-connectedness on $M$ can rule
out the possibility of $c_1(M)=-(n+1)g$. This can follows from
equality case of the following inequality due to Yau (\cite{yau}):

\begin{theorem}[Yau's Chern number inequality, negative case]
Suppose $M$ is an $n$-dimensional compact K\"{a}hler manifold with
$c_1(M)<0$. Then we have the following Chern number inequality
\be\label{chernnumber}\frac{2(n+1)}{n}(-c_1)^{n-2}
c_2\geq(-c_1)^n,\ee where the equality holds if and only if $M$ has
constant holomorphic sectional curvature, i.e., $M$ is
holomorphically covered by the unit ball in $\mathbb{C}^n$.
\end{theorem}

\emph{Completion of proof of Theorem \ref{refined}}.

Now suppose $n$ is even and $c_1(M)=-(n+1)g$. Thus means $c_1(M)<0$
and the inequality (\ref{chernnumber}) holds for $M$. The first
Pontrjagin class of $M$, $p_1(M)$, is equal to $(n+1)g^2$ as the
total Pontrjagin class $p(M)=(1+g^2)^{n+1}$. Recall that for a
complex manifold $M$ we have the basic fact
$p_1(M)=c_1^2(M)-2c_2(M)$. Then $c_2(M)=\frac{n(n+1)}{2}g^2$ and so
$M$ satisfies the equality case of (\ref{chernnumber}),
contradicting to the simply-connectedness of $M$.

\subsection{Proof of Theorem \ref{value4}}
We shall adopt the same strategy as in \cite[p. 147]{LW} to prove
Theorem \ref{value4}.

If $M$ has the same integral chomology ring as that of
$\mathbb{C}P^4$, then the Chern number $c_1c_3$ of $M$ equals to
$50$ by \cite[Corollary 2.5]{LW}. This, together with the fact that
$c_1(M)\leq 5g$ from Theorem \ref{KOM}, tells us that the possible
values of $c_1(M)$ are \be\label{possiblevalue} \pm
g,~\pm2g,~\pm5g,~-10g,~-25g,~-50g.\ee

We also know from last subsection that the Todd genus
$\text{td}(M)=1$. By the formula of Todd genus in terms of Chern
numbers for $n=4$ (\cite[p.14]{Hi}) we have
$$1=\text{td}(M)=\frac{1}{720}(-c_4+c_1c_3+3c_2^2+4c_1^2c_2-c_1^4).$$
Note that the top Chern number $c_4=5$ as it equals to the Euler
number of $M$. Combining this with $c_1c_3=50$ yields the following
relation \be\label{equation}3c_2^2+4c_1^2c_2+(-c_1^4-675)=0.\ee By
the abuse of notation we may view $c_1$ and $c_2$ as integers in
(\ref{equation}) via the identifications
$H^2(M;\mathbb{Z})=\mathbb{Z}g$ and
$H^4(M;\mathbb{Z})=\mathbb{Z}g^2$. Thus, if we view (\ref{equation})
as a quadric equation of $c_2$, we have \be\label{discriminant}
c_2=\frac{-4c_1^2\pm\sqrt{7c_1^4+2025}}{3}.\ee

It can be checked directly that the only values among
(\ref{possiblevalue}) which make $c_2$ in (\ref{discriminant})
integral are $c_1=\pm5g$. In this case $c_2=10g^2$. The possibility
$(c_1,c_2)=(-5g,10g^2)$ still satisfies the equality case in
(\ref{chernnumber}) and thus can be ruled out by the
simply-connected hypothesis as before. This means the only
possibility is $c_1(M)=5g$, which completes the proof.

\begin{remark}
The only difference between this proof and \cite{LW} is that, under
the assumption of homotopy equivalence made in \cite{LW}, the parity
of $c_1$ must be the same as that of $n+1$ (in the case of $n=4$
$c_1$ must be odd) as its modulo two reduction is exactly the second
Stiefel-Whitney class, which is an invariant under homotopy
equivalence due to the classical Wu formula. So the possible values
considered in \cite{LW} are smaller than ours in
(\ref{possiblevalue}). Fortunately, all the other values in
(\ref{possiblevalue}) make the discriminant $7c_1^4+2025$ in
(\ref{equation}) square-free and thus can still be ruled out. So the
strategy in \cite{LW} can be carried over to deal with Theorem
\ref{value4}.
\end{remark}
\section*{Acknowledgements}
Part of the work was done when I visited Taida Institute for
Mathematical Sciences at National Taiwan University in July and
August 2015. I would like to thank the institute and Professor
Chang-Shou Lin for their hospitality. I also would like to thank
Professor Valentino Tosatti for his interest and useful comments on
the content of this note, and Dr. Zhaohu Nie for some useful
discussions during our stay in Taiwan.

\end{document}